\documentclass[preprint,review,10pt,round,authoryear]{elsarticle}\usepackage[]{graphicx}\usepackage[]{color}
\makeatletter
\def\maxwidth{ %
  \ifdim\Gin@nat@width>\linewidth
    \linewidth
  \else
    \Gin@nat@width
  \fi
}
\makeatother

\definecolor{fgcolor}{rgb}{0.345, 0.345, 0.345}

\usepackage{framed}
\makeatletter
 {\par\unskip\endMakeFramed%
 \at@end@of@kframe}
\makeatother

\definecolor{shadecolor}{rgb}{.97, .97, .97}
\definecolor{messagecolor}{rgb}{0, 0, 0}
\definecolor{warningcolor}{rgb}{1, 0, 1}
\definecolor{errorcolor}{rgb}{1, 0, 0}
\newenvironment{knitrout}{}{} 

\usepackage{alltt}
\usepackage[margin=1in]{geometry}

\usepackage{amsthm}
\usepackage{amssymb}
\usepackage{amsmath}
\usepackage{lineno}
\usepackage{algorithm}
\usepackage{algorithmicx}
\usepackage{algpseudocode}
\usepackage{graphicx}
\usepackage{subfigure}
\usepackage{hyperref}
\usepackage{color}
\usepackage{multirow}
\usepackage{float}

\hypersetup{
     colorlinks = true,
}

\theoremstyle{plain}
\newtheorem{thm}{Theorem}

\newtheorem{prop}{Proposition}
\newtheorem{cor}{Corollary}
\theoremstyle{definition}

\theoremstyle{remark}
\newtheorem*{rem}{Remark}

\journal{\ }
\IfFileExists{upquote.sty}{\usepackage{upquote}}{}
\begin{document}
\begin{frontmatter}

\title{Two-stage robust optimization for orienteering problem with stochastic weights}

\author[institute1,institute3]{Ke Shang}
\author[institute2]{Felix T.S. Chan}
\author[institute1]{Stephen Karungaru\corref{cor}}\cortext[cor]{Department of Information Science and Intelligent Systems, The University of Tokushima, Japan}\ead{karunga@is.tokushima-u.ac.jp}
\author[institute1]{Kenji Terada}
\author[institute3]{Zuren Feng}
\author[institute3]{Liangjun Ke}

\address[institute1]{Department of Information Science and Intelligent Systems, The University of Tokushima, Japan}
\address[institute3]{State Key Laboratory for Manufacturing Systems Engineering, Xi'an Jiaotong University, Xi'an, China}
\address[institute2]{Department of Industrial and Systems Engineering, The Hong Kong Polytechnic University, Hung Hom, Hong Kong}

\begin{abstract}
In this paper, the two-stage orienteering problem with stochastic weights (OPSW) is considered, where the first-stage problem is to plan a path under the uncertain environment and the second-stage problem is recourse action to make sure that the length constraint is satisfied after the uncertainty is realized. Two recourse models are introduced based on two different uncertainty realization ways, one is based on sequentially realized weights which leads to the recourse model proposed by \cite{evers2014two} and the other is based on concurrently realized weights which leads to a new recourse model with less variables and less constraints and is computationally more efficient. Subsequently two two-stage robust models are introduced for OPSW based on the two different recourse models, and the relationships between the two-stage robust models and their corresponding static robsut models are investigated. Theoretical conclusions are drawn which show that the two-stage robust models are equivalent to their corresponding static robust models with the box uncertainty set defined, and the two two-stage robust models are also equivalent to each other even though they are based on different recourse models. A case study is presented by comparing the two-stage robust models with an one-stage robust model for OPSW. The numerical results of the comparative studies show the effectiveness and superiority of the proposed two-stage robust models for dealing with the two-stage OPSW. 
\end{abstract}


\end{frontmatter}

\newpage
\section{Introduction}

The orienteering problem (OP) is a routing problem which has been widely studied in the past few decades. It was first introduced by \cite{golden1987orienteering} and has been developed in terms of the problem variants, solution algorithms and applications. The original OP aims at planning a path which starts and ends at the depot location, and visits a subset of nodes in order to maximize the total collected score while the length of the path cannot exceed a predefined budget, and each node can only be visited at most one time. The OP has a wide practical application background. A few examples such as Unmanned Aerial Vehicle (UAV) mission planning (\cite{mufalli2012simultaneous,evers2014robust}), tourist trip design problem (\cite{vansteenwegen2007mobile,gavalas2014survey}) and mobile crowdsourcing problem (\cite{howe2008crowdsourcing,yuen2011survey}). More detailed surveys on OP are given in \cite{vansteenwegen2011orienteering,gunawan2016orienteering}.

The stochastic orienteering problem (SOP) is a variant of OP, which assumes that some parameters in OP are stochastic and uncertain such as the score associated with each node and the weight (or distance) associated with each arc. SOP is more appropriate than OP in practical situations. For example, in a practical environment traffic congestion affects the travel time between nodes. \cite{ilhan2008orienteering} first considered uncertainties in the score of nodes and the resulting SOP is called OP with stochastic profits (OPSP). \cite{campbell2011orienteering,evers2014two} considered uncertainties in the travel and service times and the resulting SOP is called OP with stochastic travel and service times (OPSTS) or OP with stochastic weights (OPSW). Other variants include the dynamic stochastic OP (DSOP) with stochastic time-dependent travel times (\cite{lau2012dynamic,varakantham2013optimization}) and the stochastic OPTW (SOPTW) with stochastic waiting time (\cite{zhang2014priori}).

In this paper, we focus on the OPSW where the uncertainties lie in the weights of the arcs. Some works on OPSW have been done in recent years. \cite{campbell2011orienteering} considered the OPSTS in which a penalty is incurred if a commitment to a node is made but not completed. A variant of VNS for the OPSTS is proposed and its performance is evaluated by comparing with a dynamic programming (DP) approach. \cite{evers2014two} introduced a two-stage stochastic programming model for the OPSW. The first-stage problem is to plan a path. The second-stage problem is a recourse action after the uncertain weights realized, which aborts the execution of the first-stage path and enforces a direct return to the depot. The objective is to maximize the first-stage path score plus the expected loss of the score due to the recourse action. They presented a Sample Average Approximation (SAA) approach and an OPSW heuristic to solve the problem and the performance of the two approaches were evaluated. \cite{evers2014robust} applied the robust optimization (RO) methodology to build robust models for UAV mission planning with uncertain fuel usage between targets, which is an OPSW in nature. The performance of the robust models are studied in terms of the different uncertainty sets and the feasibility of the robust solutions.

Inspired by the recourse model proposed by \cite{evers2014two}, we consider the two-stage OPSW, i.e. OPSW with recourse action. The first-stage problem is to plan a path with the stochastic weights unrevealed. The second-stage problem is a recourse action to avoid the violation of the length budget after the uncertainty realized. The recourse action is to abort the execution of the first-stage path and enforce a direct return to the depot. This kind of recourse action is necessary especially in the UAV mission planning. The UAV has to return to the depot safely in the uncertain environment. We notice that the realization way of the uncertainty is not unique. For example, the uncertain weights of the first-stage path can be realized sequentially during the path execution, or the uncertain weights of the first-stage path can be realized concurrently at the beginning of the path execution. Different realization ways for the uncertainty can lead to different recourse models. We thus define two realization ways of the uncertainty: \emph{Sequential realization} and \emph{Concurrent realization}. The \emph{Sequential realization} way leads to the recourse model proposed by \cite{evers2014two}, and the \emph{Concurrent realization} way leads to a new model with less variables and less constraints, which is computationally more attractive.

Two-stage robust optimization (RO), also known as adjustable RO and can be extended to the multi-stage situation, was initially introduced by \cite{ben2004adjustable}. Compared with traditional one-stage RO, two-stage RO divides the decision variables into ``here and now" decisions and ``wait and see" decisions, which is more flexible and is suitable for modeling two-stage problems. It has been successfully applied to different applications such as unit commitment (\cite{an2015exploring,wang2016two}), network flow (\cite{atamturk2007two,ordonez2007robust}) and portfolio optimization (\cite{takeda2008adjustable}). In this paper, we apply the two-stage RO paradigm to the two-stage OPSW for the first time and introduce two two-stage RO models based on two different recourse models. The two-stage RO models introduced in this paper are with binary recourse decisions and this kind of problem has largely resisted solution so far (\cite{hanasusanto2015k}). Computing an optimal adjustable robust solution is often intractable since it requires computing a solution for all possible realizations of the uncertainties (\cite{feige2007robust}). Instead of solving the two-stage RO model directly, \cite{bertsimas2015tight} studied the performance of the static solutions for two-stage adjustable robust linear optimization problems and presented a tight characterization of the conditions under which a static robust solution is optimal for the two-stage robust problem. From this point of view, we introduce two static robust models for the OPSW which correspond to the two two-stage robust models respectively, and study their performance and the relationships with the two-stage robust models. We prove that with the \emph{box uncertainty set} defined, the two-stage robust models are equivalent to their corresponding static robust models, and the two two-stage robust models are also equivalent to each other even though they are based on different recourse models. These conclusions we obtained indicate that the two-stage robust models for OPSW can be solved to optimality by solving their corresponding static robust models, and also we only need to use one static robust model, which is based on the second recourse model, to deal with two different uncertainty realization ways.

The main contributions of this paper are summarized as follows:
\begin{enumerate}
\item Two recourse models are presented for the two-stage OPSW: one is the recourse model with \emph{Sequential realization} and the other is the recourse model with \emph{Concurrent realization}.
\item Two two-stage robust models are presented for the first time for the OPSW based on the two different recourse models.
\item Three theorems are established which show the equivalence between the two-stage robust models and their corresponding static robust models.
\item The two-stage robsut models for OPSW are evaluated numerically by comparing with one-stage robust model for OPSW.
\end{enumerate}

The remainder of the paper is organized as follows. First the deterministic OP is described in Section \ref{section:op}. Section \ref{section:2op} describes the two-stage OPSW and introduces two recourse models with different uncertainty realization ways. Section \ref{section:2ro} introduces two two-stage robust models for OPSW and draws some theoretical conclusions of the equivalence between the two-stage robust models and their corresponding static robust models. A case study is presented in Section \ref{section:cs} and we conclude the whole paper in Section \ref{section:conclusion}.

\section{The deterministic OP}
\label{section:op}
In a deterministic OP, a set of vertices $N$ is given with $|N|$ as its cardinality. Each vertex $i\in N$ has a score $s_i$ associated with it. Denote $0$ as the depot location where $0\notin N$ and $N^+=N\cup\{0\}$. The goal is to plan a path with length limit $L$, that starts and ends at the depot and visits some vertices in order to maximize the sum of the collected scores. Each vertex is visited at most one time.

Suppose all nodes $N^+$ are on a complete graph $G=(N^+,A)$ where $A$ is the set of arcs connecting the vertices in $N^+$. The weight of each arc $(i,j)\in A$ is $d_{ij}$, representing the Euclidean distance from $i$ to $j$. Let $x_{ij}$ be a binary decision variable, where $x_{ij}=1$ if and only if arc $(i,j)$ is visited by the path, otherwise $x_{ij}=0$. An auxiliary variable $u_i$ is used to denote the position of node $i$ in the path. The formulation of the deterministic OP is as follows:

\begin{subequations}\label{eq:op1}
(DOP):\begin{alignat}{2}
    \text{maximize}\qquad & \sum_{i\in N}s_i\sum_{j\in N^+\setminus\{i\}}x_{ij}\\
    \text{subject to }\qquad & \sum_{(i,j)\in A}d_{ij} x_{ij} \leq L\label{opc:1}\\
    & \sum_{i\in N}x_{0i}=\sum_{i\in N}x_{i0}=1\label{opc:2}\\
    & \sum_{i\in N^+\setminus\{j\}}x_{ij}=\sum_{i\in N^+\setminus\{j\}}x_{ji}\leq1, \forall j\in N\label{opc:3}\\
    & u_i-u_j+1\leq(1-x_{ij})|N|, \forall i,j\in N\label{opc:4}\\
    & 1\leq u_i\leq|N|, \forall i\in N\label{opc:5}\\
    & x_{ij}\in \{0,1\}, \forall (i,j)\in A\label{opc:6}
    \end{alignat}
  \end{subequations}
Constraint \eqref{opc:1} is the path length constraint. Constraint \eqref{opc:2} guarantees that the path starts and ends at the depot. Constraint \eqref{opc:3} is the flow conservation constraint ensuring that a vertex is visited at most once. Constraint \eqref{opc:4} ensures the connectivity of the path. Constraint \eqref{opc:5} and \eqref{opc:6} are the boundary and integrality constraints on the auxiliary variables and decision variables respectively. 

\section{The two-stage orienteering problem with stochastic weights}
\label{section:2op}

Suppose the weight of each arc $(i,j)$ is stochastic and uncertain, denote the stochastic weight of arc $(i,j)$ as $\tilde{d}_{ij}$. In this paper, we consider $\tilde{d}_{ij}$ as a symmetrically distributed random variable on the interval $[\bar{d}_{ij}-\hat{d}_{ij},\bar{d}_{ij}+\hat{d}_{ij}]$, where $\bar{d}_{ij}$ is the expected value of $\tilde{d}_{ij}$ and $\hat{d}_{ij}$ is the maximum deviation of $\tilde{d}_{ij}$ from its expected value. For simplicity and convenience, we use $d_{ij}$ to denote the realizations of $\tilde{d}_{ij}$.

We consider the two-stage OPSW, i.e. OPSW with recourse action. In the two-stage OPSW, the first-stage problem is to plan a path with the stochastic weights unrevealed. Due to the randomicity and uncertainty of the stochastic weights, the first-stage path may violate constraint \eqref{opc:1} after the uncertainty is realized. So the second-stage problem is a recourse action to avoid constraint violation after the uncertainty is realized. The recourse action is to abort the execution of the first-stage path and enforce a direct return to the depot.

The ways that the uncertainty realized are not unique. Different realization ways for the uncertainty will lead to different recourse models. We introduce two realization ways for the uncertainty in two-stage OPSW: \emph{Sequential realization} and \emph{Concurrent realization}. \emph{Sequential realization} means that the stochastic weights of the first-stage path are realized sequentially during the path execution. For example, the first-stage path is executed to node $i$ and the next node is $j$, then the stochastic weight $\tilde{d}_{ij}$ is realized and the stochastic weights of all other unvisited arcs remain unrevealed. \emph{Concurrent realization} means that all the stochastic weights of the first-stage path are realized concurrently at the beginning of the path execution and the stochastic weights of all other arcs remain unrevealed.

Based on the above two realization ways for uncertainty, we now present two recourse models for two-stage OPSW.

\subsection{Recourse model with Sequential realization}

The recourse model with sequentially realized weights was initially introduced by \cite{evers2014two}. In this model, the uncertain weights of the first-stage path are realized sequentially during the path execution. The uncertainty realization rule is: suppose the first-stage path is executed to node $i$ and the next node is $j$, then the stochastic weight $\tilde{d}_{ij}$ is realized and the stochastic weights of all other unvisited arcs remain unrevealed. Then the recourse action is to abort the execution of the first-stage path and enforce a direct return to the depot from node $i$ at the moment that the remaining length budget is insufficient to support a visit to the next node $j$ plus the expected return length from the next node $j$ to the depot. \cite{evers2014two} assumed that a certain amount of extra length budget is available to cover the maximum deviation from the expected length on any of the arcs to the depot, this safety stock not being part of the length limit $L$ used in the model.

Denoting the first-stage path as vector $\mathbf{x}$ which contains all $x_{ij}$, and the weight realizations as vector $\mathbf{d}$ which contains all $d_{ij}$. Let $x_{ijk}$ be a binary variable, $x_{ijk}=1$ if arc $(i,j)$ is the $k$th arc in the first-stage path, otherwise $x_{ijk}=0$; let $y_i$ be a binary variable, $y_i=1$ if node $i$ is in the first-stage path but cannot be reached as a result of the recourse action, otherwise $y_i=0$; let $z_k$ be a binary variable, $z_k=1$ if the $k$th node in the first-stage path cannot be reached as a result of the recourse action, otherwise $z_k=0$. With the first-stage path $\mathbf{x}$ and the weight realizations $\mathbf{d}$, the recourse problem of the two-stage OPSW with sequentially realized weights is formulated as follows:

\begin{subequations}\label{eq:rs}
(Recourse-Sequential):
    \begin{alignat}{2}
    RS(\mathbf{x},\mathbf{d})=\max \qquad & -\sum_{i\in N}s_i y_i\label{eq:re1}\\
    \text{subject to }\qquad & x_{0j1}\geq x_{0j},\forall j\in N\label{2rm1:1}\\
    & x_{ijk}\geq x_{ij} + \sum_{l\in N^+}x_{li(k-1)} -1,\forall i,j\in N, k=1,...,|N|\label{2rm1:2}\\
    & \sum_{k=1}^{K}\sum_{(i,j)\in A}d_{ij}x_{ijk} + \sum_{(i,j)\in A}\bar{d}_{j0}x_{ijK} \leq L + Mz_{K},\forall K=1,...,|N|\label{2rm1:3}\\
    & z_{k}\geq z_{k-1},\forall k=2,...,|N|\label{2rm1:4}\\
    & y_{j}\geq \sum_{i\in N^+}x_{ijk}+z_{k}-1,\forall j\in N, k=1,...,|N|\label{2rm1:5}\\
    & x_{ijk}\in \{0,1\}, \forall i,j\in N^+, k=1,...,|N|\label{2rm1:6}\\
    &  y_{i}\in\{0,1\},\forall i\in N\label{2rm1:7}\\
    &z_{k}\in\{0,1\},\forall k=1,...,|N|\label{2rm1:8}
    \end{alignat}
  \end{subequations}
where the objective function \eqref{eq:re1} is to minimize the loss in the collected score as a result of the recourse action. Constraint \eqref{2rm1:1} identifies the first arc in the path. Constraint \eqref{2rm1:2} identifies the order of the other arcs in the path. Constraint \eqref{2rm1:3} determines the nodes of the first-stage path $\mathbf{x}$ that can and cannot be reached based on $\mathbf{d}$, where $M$ is a sufficiently large number. Constraint \eqref{2rm1:4} makes sure that all nodes in the path after the first node that cannot be reached, cannot be reached either. Constraint \eqref{2rm1:5} identifies the nodes in the first-stage path that cannot be reached, based on their indexes. A detailed explanation of the model is given in \cite{evers2014two}.

With the first-stage path $\mathbf{x}$ and the weight realizations $\mathbf{d}$, the objective value of the sequential recourse problem can be calculated not only by solving the \hyperref[eq:rs]{Recourse-Sequential} model, but also by an efficient Forward Checking algorithm which is described in Algorithm \ref{algorithm:1}. With Algorithm \ref{algorithm:1}, the objective value of the \hyperref[eq:rs]{Recourse-Sequential} model can be obtained in time $\mathcal{O}(n)$ where $n$ is the number of nodes in the first-stage path.

\renewcommand{\algorithmicrequire}{\textbf{Input:}} 
\renewcommand{\algorithmicensure}{\textbf{Output:}} 
\begin{algorithm}[h]
\caption{Forward Checking algorithm for \hyperref[eq:rs]{Recourse-Sequential} model}
\begin{algorithmic}[1]
\Require
The first-stage path $\mathbf{x}$ and the weight realizations $\mathbf{d}$
\Ensure
The objective value of $RS(\mathbf{x},\mathbf{d})$
\State Denote the first-stage path $\mathbf{x}$ as a node sequence $(v_0,v_1,...,v_n,v_0)$ where $v_0$ is the depot and $v_k$ is the $k$th node.
\State $violation=False$
\For{$k=1$ to $n$}
\State $Length=dist(v_0,v_1,...,v_k)+\bar{d}_{v_kv_0}$ where $dist(v_0,v_1,...,v_k)=d_{v_0v_1}+...+d_{v_{k-1}v_k}$.
\If{$Length>L$}
\State $violation=True$
\State break for loop
\EndIf
\EndFor
\If{$violation==True$}
\State $Loss=s_{v_k}+s_{v_{k+1}}+...+s_{v_n}$
\Else
\State $Loss=0$
\EndIf
\State $RS(\mathbf{x},\mathbf{d})=-Loss$
\end{algorithmic}
\label{algorithm:1} 
\end{algorithm}

\subsection{Recourse model with Concurrent realization}

We now introduce a recourse model with concurrently realized weights. In this model, all the stochastic weights of the first-stage path are realized concurrently at the beginning of the path execution and the stochastic weights of all other arcs remain unrevealed. Then the recourse action is to find a node $i$ in the first-stage path and enforce a direct return to the depot from node $i$ so that the length of the subpath from the depot to node $i$ plus the expected return length from node $i$ to the depot is within the length limit $L$ and the loss in the collected score is minimized. Here, we also assume that a certain amount of extra length budget is available to cover the maximum deviation from the expected length on any of the arcs to the depot. This safety stock is not part of the length limit $L$ used in the model.

We introduce a new binary variable $y_{ij}$; $y_{ij}=1$ if arc $(i,j)$ is in the first-stage path but is cancelled by the recourse action, $y_{ij}=0$ if arc $(i,j)$ is in the first-stage path and is not cancelled by the recourse action, or arc $(i,j)$ is not in the first-stage path. Then the recourse problem of the two-stage OPSW with concurrently realized weights can be formulated as follows:

\begin{subequations}\label{eq:rc}
(Recourse-Concurrent):
    \begin{alignat}{2}
    RC(\mathbf{x},\mathbf{d})=\max \qquad & -\sum_{j\in N}s_j\sum_{i\in N^+\setminus\{j\}}y_{ij}\label{2rm2:1}\\
    \text{subject to }\qquad &y_{ij}\leq x_{ij},\forall (i,j)\in A \label{2rm2:2}\\
    &\sum_{i\in N^+\setminus\{j\}}y_{ij}\leq \sum_{k\in N^+\setminus\{j\}}y_{jk},\forall j\in N\label{2rm2:4}\\
    & \sum_{(i,j)\in A}d_{ij}x_{ij}-\sum_{(i,j)\in A}d_{ij}y_{ij}+\sum_{j\in N}\left(\sum_{k\in N^+\setminus\{j\}}y_{jk}-\sum_{i\in N^+\setminus\{j\}}y_{ij}\right)\bar{d}_{j0}\leq L\label{2rm2:5}\\
    &y_{ij}\in\{0,1\},\forall (i,j)\in A\label{2rm2:6}
    \end{alignat}
  \end{subequations}
where the objective function \eqref{2rm2:1} is to minimize the loss in the collected score as a result of the recourse action. Constraint \eqref{2rm2:2} ensures that the cancelled arcs are from the first-stage path. Constraint \eqref{2rm2:4} ensures that the cancelled arcs compose a subpath of the first-stage path starting at a vertex of the first-stage path and ending at the depot. Constraint \eqref{2rm2:5} ensures that the modified path after the recourse action is within the length limit.

With the first-stage path $\mathbf{x}$ and the weight realizations $\mathbf{d}$, the objective value of the concurrent recourse problem can be calculated not only by solving the \hyperref[eq:rc]{Recourse-Concurrent} model, but also by an efficient Backward Checking algorithm which is described in Algorithm \ref{algorithm:2}. With Algorithm \ref{algorithm:2}, the objective value of the \hyperref[eq:rc]{Recourse-Concurrent} model can be obtained in time $\mathcal{O}(n)$ where $n$ is the number of nodes in the first-stage path.

\renewcommand{\algorithmicrequire}{\textbf{Input:}} 
\renewcommand{\algorithmicensure}{\textbf{Output:}} 
\begin{algorithm}[h]
\caption{Backward Checking algorithm for \hyperref[eq:rc]{Recourse-Concurrent} model}
\begin{algorithmic}[1]
\Require
The first-stage path $\mathbf{x}$ and the weight realizations $\mathbf{d}$
\Ensure
The objective value of $RC(\mathbf{x},\mathbf{d})$
\State Denote the first-stage path $\mathbf{x}$ as a node sequence $(v_0,v_1,...,v_n,v_0)$ where $v_0$ is the depot and $v_k$ is the $k$th node.
\State $violation=True$
\For{$k=n$ to $1$}
\State $Length=dist(v_0,v_1,...,v_k)+\bar{d}_{v_kv_0}$ where $dist(v_0,v_1,...,v_k)=d_{v_0v_1}+...+d_{v_{k-1}v_k}$.
\If{$Length\leq L$}
\State $violation=False$
\State break for loop
\EndIf
\EndFor
\If{$k\neq n$ and $violation==False$}
\State $Loss=s_{v_{k+1}}+...+s_{v_n}$
\ElsIf{$k\neq n$ and $violation==True$}
\State $Loss=s_{v_1}+...+s_{v_n}$
\Else
\State $Loss=0$
\EndIf
\State $RC(\mathbf{x},\mathbf{d})=-Loss$
\end{algorithmic}
\label{algorithm:2} 
\end{algorithm}
%

\section{Two-stage robust optimization for OPSW}
\label{section:2ro}

In this section, we apply the two-stage RO methodology to model the two-stage OPSW. In the two-stage OPSW, the first-stage ``here and now'' decisions are the binary decision variables $x_{ij}$ described in Section \ref{section:op}. The second-stage ``wait and see'' decisions are the binary decision variables $y_i$ and $z_k$ in the \hyperref[eq:rs]{Recourse-Sequential} model or the binary decision variables $y_{ij}$ in the \hyperref[eq:rc]{Recourse-Concurrent} model.

In the two-stage RO for OPSW, an uncertainty set needs to be defined for the stochastic weights. We consider the \emph{box uncertainty set} which is defined by the $\infty$-norm of the uncertain vector. The reasons that we choose the \emph{box uncertainty set} for the two-stage RO for OPSW are as follows:
\begin{enumerate}
\item It is simple compared with the polyhedral uncertainty set and the ellipsoidal uncertainty set which is defined by 1-norm and 2-norm respectively (\cite{bertsimas2004robust,ben2009robust}), and the derived robust counterpart has the same computational complexity as the original model.   
\item With the \emph{box uncertainty set}, we can draw some interesting conclusions in the following subsections which describe the equivalence between the two-stage robust models and their corresponding static robust models.
\end{enumerate}

Without loss of generality, the \emph{box uncertainty set} $\mathcal{U}$ for the stochastic weights is defined as follows:
\begin{equation}
\label{uncertainty:set1}
\mathcal{U}=\left\{\mathbf{d}\in\mathbb{R}^{M}:d_{ij}=\bar{d}_{ij}+\zeta_{ij}\hat{d}_{ij},\forall i,j\in N^+,\zeta\in \mathcal{Z}\right\}
\end{equation}
where $\mathbf{d}$ is a $M$ dimensional vector with $M=|N^+|\times|N^+|$, $\zeta\in\mathbb{R}^M$ is the vector of primitive uncertainties, and $\mathcal{Z}$ is a convex set which is defined as follows:
\begin{equation}
\label{uncertainty:set2}
\mathcal{Z}=\left\{\zeta\in\mathbb{R}^{M}: \left \|\zeta  \right \|_{\infty}\leq \Theta\right\}
\end{equation}
where $\Theta\in [0,1]$ is the parameter controlling the size of $\mathcal{Z}$.

Next, we introduce two two-stage RO models for OPSW based on the \hyperref[eq:rs]{Recourse-Sequential} model and the \hyperref[eq:rc]{Recourse-Concurrent} model respectively.

\subsection{Two-stage robust model for OPSW with \hyperref[eq:rs]{Recourse-Sequential} model}

Based on the \hyperref[eq:rs]{Recourse-Sequential} model and the two-stage RO paradigm, we introduce the following two-stage RO model for OPSW:

\begin{subequations}\label{2ro:1}
(Two-stage-Sequential):
    \begin{alignat}{2}
    \text{maximize}\qquad & \sum_{i\in N}s_i\sum_{j\in N^+\setminus\{i\}}x_{ij} + \min_{\mathbf{d}\in\mathcal{U}}RS(\mathbf{x},\mathbf{d})\\
    \text{subject to }\qquad &\sum_{(i,j)\in A}(\bar{d}_{ij}-\hat{d}_{ij}) x_{ij} \leq L\label{2ro:2}\\
    &\eqref{opc:2}-\eqref{opc:6}
    \end{alignat}
  \end{subequations}
where $RS(\mathbf{x},\mathbf{d})$ is the \hyperref[eq:rs]{Recourse-Sequential} model and $\mathcal{U}$ is the box uncertainty set. Constraint \eqref{2ro:2} is the length limit on the first-stage path. Without constraint \eqref{2ro:2}, the first-stage path can be arbitrarily long providing there exists unvisited nodes, and these nodes can be included in the first-stage path even some nodes in any particular case cannot be reached. By adding constraint \eqref{2ro:2}, we limit the length of the first-stage path in the most optimistic situation, i.e. all arc weights equal to their minimum values. With this constraint, the size of the solution space can be reduced while the problem optimality is maintained.

The two-stage robust model for OPSW introduced above is an 0-1 integer programming problem with 0-1 integer recourse. Next, we present its corresponding static robust model in which the second-stage ``wait and see'' decisions become ``here and now''. The corresponding static robust model of the \hyperref[2ro:1]{Two-stage-Sequential} model is formulated as follows:

\begin{subequations}\label{sro:1}
(Static-Sequential):
    \begin{alignat}{2}
    \text{maximize}\qquad & \sum_{i\in N}s_i\sum_{j\in N^+\setminus\{i\}}x_{ij} -\sum_{i\in N}s_i y_i\\
    \text{subject to }\qquad &\eqref{opc:2}-\eqref{opc:6},\eqref{2ro:2}\\
    & \eqref{2rm1:1}-\eqref{2rm1:2},\eqref{2rm1:4}-\eqref{2rm1:8}\\
    & \sum_{k=1}^{K}\sum_{(i,j)\in A}d_{ij}x_{ijk} + \sum_{(i,j)\in A}\bar{d}_{j0}x_{ijK} \leq L + Mz_{K},\forall K=1,...,|N|,\mathbf{d}\in\mathcal{U}\label{sro:2}
    \end{alignat}
  \end{subequations}
  
In the above static robust model, the second-stage decision variables $y_i$ and $z_k$ are ``here and now'' and do not depend on the realizations of uncertain $\mathbf{d}$. Both first-stage decisions $x_{ij}$ and second-stage decisions $y_i$ and $z_k$ are selected before the uncertain $\mathbf{d}$ is known. An optimal static robust solution to \hyperref[sro:1]{Static-Sequential} can be computed efficiently with the box uncertainty set $\mathcal{U}$. What interests us is the relationship between the two-stage robust model \hyperref[2ro:1]{Two-stage-Sequential} and its corresponding static robust model \hyperref[sro:1]{Static-Sequential}. 

In the following, we establish a theorem which describes the equivalence of the \hyperref[2ro:1]{Two-stage-Sequential} model and the \hyperref[sro:1]{Static-Sequential} model.
  
\begin{thm}
\label{thm:1}
The two-stage robust model \hyperref[2ro:1]{Two-stage-Sequential} and its corresponding static robust model \hyperref[sro:1]{Static-Sequential} are equivalent.
\end{thm}
\begin{proof}
It is clear that the optimal solution of the static robust model \hyperref[sro:1]{Static-Sequential} is a feasible solution of the two-stage robust model \hyperref[2ro:1]{Two-stage-Sequential}. All we need to show is that the optimal solution of the two-stage robust model \hyperref[2ro:1]{Two-stage-Sequential} is a feasible solution of the static robust model \hyperref[sro:1]{Static-Sequential}. 

We prove by apagoge. Denote $(x_{ij}^*,x_{ijk}^*,y_i^*,z_k^*)$ as the optimal solution of the two-stage robust model \hyperref[2ro:1]{Two-stage-Sequential}. Suppose the optimal solution is infeasible for the static robust model \hyperref[sro:1]{Static-Sequential}, which means 
\begin{equation}
\exists \mathbf{d}'\in\mathcal{U}\text{ and }K', \sum_{k=1}^{K'}\sum_{(i,j)\in A}d'_{ij}x_{ijk}^* + \sum_{(i,j)\in A}\bar{d}_{j0}x_{ijK'}^* > L + Mz_{K'}^*\label{1}
\end{equation}

Because $M$ is a sufficiently large number, so the above condition is only satisfied by $z_{K'}^*=0$. This means that the $K'$th node in the first-stage path of the two-stage robust model \hyperref[2ro:1]{Two-stage-Sequential} is reachable, but this node is unreachable under the context of the static robust model \hyperref[sro:1]{Static-Sequential}.

Now we consider the second-stage problem $RS(\mathbf{x}^*,\mathbf{d}')$ where $\mathbf{x}^*$ is the first-stage optimal solution, and denote the optimal solution as $(y'_i,z'_k)$. Then for $K'$,
\begin{equation}
\sum_{k=1}^{K'}\sum_{(i,j)\in A}d'_{ij}x_{ijk}^* + \sum_{(i,j)\in A}\bar{d}_{j0}x_{ijK'}^* \leq L + Mz'_{K'}\label{2}
\end{equation}

Comparing constraints \eqref{1} and \eqref{2}, it is clear that $z'_{K'}$ must be 1. This means the $K'$th node in the first-stage path is unreachable with $\mathbf{d}'$. Denote the $K'$th node in the first-stage path as node $j$, then the second-stage optimal value $RS(\mathbf{x}^*,\mathbf{d}')\leq -\sum_{i\in N}s_i y_i^*-s_{j}$. Because we are optimizing $\min_{\mathbf{d}\in\mathcal{U}}RS(\mathbf{x}^*,\mathbf{d})$, $(y_i^*,z_k^*)$ is not the optimal second-stage solution, so this is a contradiction. Thus the hypothesis cannot be established, which means $(x_{ij}^*,x_{ijk}^*,y_i^*,z_k^*)$ is feasible for the static robust model \hyperref[sro:1]{Static-Sequential}.

We conclude that the optimal solution of the two-stage robust model \hyperref[2ro:1]{Two-stage-Sequential} is a feasible solution of the static robust model \hyperref[sro:1]{Static-Sequential}. Because the two models have the same objective function value with the same solutions, then the optimal solution of the two-stage robust model \hyperref[2ro:1]{Two-stage-Sequential} is also the optimal solution of the static robust model \hyperref[sro:1]{Static-Sequential}, this implies the two models are equivalent.
\end{proof}

\begin{rem}
The proof of Theorem \ref{thm:1} does not need the support of the \emph{box uncertainty set}. We can still draw this conclusion even if the uncertainty set $\mathcal{U}$ is an arbitrary uncertainty set.
\end{rem}

Based on Theorem \ref{thm:1}, the two-stage robust model \hyperref[2ro:1]{Two-stage-Sequential} can be solved to optimality by solving its corresponding static robust model \hyperref[sro:1]{Static-Sequential}. Comparing the static robust model \hyperref[sro:1]{Static-Sequential} with the original deterministic OP model \hyperref[eq:op1]{DOP}, many new integer variables are added which makes the static robust model \hyperref[sro:1]{Static-Sequential} computationally expensive. \cite{evers2014two} proved that the relaxation model with $0\leq x_{ijk}\leq 1$ and $0\leq y_i\leq 1$ of the second-stage problem $RS(\mathbf{x},\mathbf{d})$ is equivalent to the original $RS(\mathbf{x},\mathbf{d})$, and the resulting relaxation model provides a substantial decrease in the computation time. This conclusion can be easily applied to the static robust model \hyperref[sro:1]{Static-Sequential} which leads to the following proposition:

\begin{prop}
The relaxation model with $0\leq x_{ijk}\leq 1$ and $0\leq y_i\leq 1$ of the static robust model \hyperref[sro:1]{Static-Sequential} is equivalent to the original static robust model \hyperref[sro:1]{Static-Sequential}.
\end{prop}
\begin{proof}
This conclusion can be drawn by following the proof way of Theorem 1 in \cite{evers2014two}.
\end{proof}

\subsection{Two-stage robust model for OPSW with \hyperref[eq:rc]{Recourse-Concurrent} model}

Based on the \hyperref[eq:rc]{Recourse-Concurrent} model and the two-stage RO paradigm, we introduce the following two-stage RO model for OPSW:

\begin{subequations}\label{2roc:1}
(Two-stage-Concurrent):
    \begin{alignat}{2}
    \text{maximize}\qquad & \sum_{i\in N}s_i\sum_{j\in N^+\setminus\{i\}}x_{ij} + \min_{\mathbf{d}\in\mathcal{U}}RC(\mathbf{x},\mathbf{d})\\
    \text{subject to }\qquad &\eqref{opc:2}-\eqref{opc:6},\eqref{2ro:2}
    \end{alignat}
  \end{subequations}
where $RC(\mathbf{x},\mathbf{d})$ is the \hyperref[eq:rc]{Recourse-Concurrent} model and $\mathcal{U}$ is the box uncertainty set. 

We also consider the corresponding static robust model of \hyperref[2roc:1]{Two-stage-Concurrent} instead of solving the two-stage robust model directly, the corresponding static robust model of \hyperref[2roc:1]{Two-stage-Concurrent} is formulated as follows:

\begin{subequations}\label{sro:2}
(Static-Concurrent):
    \begin{alignat}{2}
    \text{maximize}\qquad & \sum_{i\in N}s_i\sum_{j\in N^+\setminus\{i\}}x_{ij} -\sum_{j\in N}s_j\sum_{i\in N^+\setminus\{j\}}y_{ij}\\
    \text{subject to }\qquad &\eqref{opc:2}-\eqref{opc:6},\eqref{2ro:2}\\
    & \eqref{2rm2:2}-\eqref{2rm2:4},\eqref{2rm2:6}\\
    & \sum_{(i,j)\in A}d_{ij}x_{ij}-\sum_{(i,j)\in A}d_{ij}y_{ij}+\sum_{j\in N}\left(\sum_{k\in N^+\setminus\{j\}}y_{jk}-\sum_{i\in N^+\setminus\{j\}}y_{ij}\right)\bar{d}_{j0}\leq L,\forall \mathbf{d}\in\mathcal{U}\label{sroc:2}
    \end{alignat}
  \end{subequations}

We can readily see that the optimal solution of the static robust model \hyperref[sro:2]{Static-Concurrent} is feasible to the two-stage robust model \hyperref[2roc:1]{Two-stage-Concurrent}. With the help of the box uncertainty set, the following theorem can be established which shows that \hyperref[2roc:1]{Two-stage-Concurrent} and \hyperref[sro:2]{Static-Concurrent} are equivalent.

\begin{thm}
\label{thm:3}
The two-stage robust model \hyperref[2roc:1]{Two-stage-Concurrent} and its corresponding static robust model \hyperref[sro:2]{Static-Concurrent} are equivalent.
\end{thm}
\begin{proof}
It is clear that the optimal solution of the static robust model \hyperref[sro:2]{Static-Concurrent} is a feasible solution of the two-stage robust model \hyperref[2roc:1]{Two-stage-Concurrent}. All we need to show is that the optimal solution of the two-stage robust model \hyperref[2roc:1]{Two-stage-Concurrent} is a feasible solution of the static robust model \hyperref[sro:2]{Static-Concurrent}. 

We prove by apagoge. Denote $(x_{ij}^*,y_{ij}^*)$ as the optimal solution of the two-stage robust model \hyperref[2roc:1]{Two-stage-Concurrent}. Suppose the optimal solution is infeasible for the static robust model \hyperref[sro:2]{Static-Concurrent}, which means 
\begin{equation}
\exists \mathbf{d}'\in\mathcal{U}, \sum_{(i,j)\in A}d'_{ij}x^*_{ij}-\sum_{(i,j)\in A}d'_{ij}y^*_{ij}+\sum_{j\in N}\left(\sum_{k\in N^+\setminus\{j\}}y^*_{jk}-\sum_{i\in N^+\setminus\{j\}}y^*_{ij}\right)\bar{d}_{j0}> L\label{thm3:eq1}
\end{equation}

Denote $\mathbf{d}^{u} = \bar{\mathbf{d}}+\Theta\hat{\mathbf{d}}$, according to the definition of the box uncertainty set $\mathcal{U}$, we know that $\mathbf{d}^{u}\in\mathcal{U}$ and $\mathbf{d}^{u}\succeq \mathbf{d}'$ where $\succeq$ is the element-wise inequality. Based on inequality \eqref{thm3:eq1}, it is clear that
\begin{equation}
\sum_{(i,j)\in A}d^{u}_{ij}x^*_{ij}-\sum_{(i,j)\in A}d^{u}_{ij}y^*_{ij}+\sum_{j\in N}\left(\sum_{k\in N^+\setminus\{j\}}y^*_{jk}-\sum_{i\in N^+\setminus\{j\}}y^*_{ij}\right)\bar{d}_{j0}> L\label{thm3:1}
\end{equation}

Denote $\mathbf{d}^*\in\mathcal{U}$ as the optimal value of $\mathbf{d}$ that achieves optimal solution  $(x_{ij}^*,y_{ij}^*)$ in the two-stage robust model \hyperref[2roc:1]{Two-stage-Concurrent}, then 
\begin{equation}
\sum_{(i,j)\in A}d^*_{ij}x^*_{ij}-\sum_{(i,j)\in A}d^*_{ij}y^*_{ij}+\sum_{j\in N}\left(\sum_{k\in N^+\setminus\{j\}}y^*_{jk}-\sum_{i\in N^+\setminus\{j\}}y^*_{ij}\right)\bar{d}_{j0}\leq L
\end{equation}
and
\begin{equation}
\sum_{(i,j)\in A}d^*_{ij}x^*_{ij}-\sum_{(i,j)\in A}d^*_{ij}y'_{ij}+\sum_{j\in N}\left(\sum_{k\in N^+\setminus\{j\}}y'_{jk}-\sum_{i\in N^+\setminus\{j\}}y'_{ij}\right)\bar{d}_{j0}> L,\forall \mathbf{y}'\in \mathcal{Y}\text{ and }\mathbf{y}'\neq \mathbf{y}^*\label{thm3:eq2}
\end{equation}
where $\mathcal{Y}=\left\{ \mathbf{y}':\begin{matrix}
y'_{ij}\leq y^*_{ij},\forall i,j\\ 
\sum_{i\in N^+\setminus\{j\}}y'_{ij}\leq \sum_{k\in N^+\setminus\{j\}}y'_{jk},\forall j\\ 
y'_{ij}\in \{0,1\},\forall i,j
\end{matrix} \right\}$ is the set which contains all recourse actions with less cancelled arcs comparing with $\mathbf{y}^*$.

Based on the fact that $\mathbf{d}^{u}\succeq \mathbf{d}^*$ and using inequality \eqref{thm3:eq2}, we have
\begin{equation}
\sum_{(i,j)\in A}d^{u}_{ij}x^*_{ij}-\sum_{(i,j)\in A}d^{u}_{ij}y'_{ij}+\sum_{j\in N}\left(\sum_{k\in N^+\setminus\{j\}}y'_{jk}-\sum_{i\in N^+\setminus\{j\}}y'_{ij}\right)\bar{d}_{j0}> L,\forall \mathbf{y}'\in \mathcal{Y}\text{ and }\mathbf{y}'\neq \mathbf{y}^*\label{thm3:2}
\end{equation}

Combining inequalities \eqref{thm3:1} and \eqref{thm3:2}, we can observe that: for the second-stage problem $RC(\mathbf{x}^*,\mathbf{d}^{u})$, the recourse action needs to cancel more arcs than $\mathbf{y}^*$ to satisfy the length constraint, which means $RC(\mathbf{x}^*,\mathbf{d}^{u})< -\sum_{j\in N}s_j\sum_{i\in N^+\setminus\{j\}}y^*_{ij}$. Because we are optimizing $\min_{\mathbf{d}\in\mathcal{U}}RC(\mathbf{x}^*,\mathbf{d})$, $\mathbf{y}^*$ is not the optimal second-stage solution, so this is a contradiction. Hence, the hypothesis cannot be established, which means $(x_{ij}^*,y_{ij}^*)$ is feasible for the static robust model \hyperref[sro:2]{Static-Concurrent}.

We conclude that the optimal solution of the two-stage robust model \hyperref[2roc:1]{Two-stage-Concurrent} is a feasible solution of the static robust model \hyperref[sro:2]{Static-Concurrent}. Because the two models have the same objective function value with the same solutions, then the optimal solution of the two-stage robust model \hyperref[2roc:1]{Two-stage-Concurrent} is also the optimal solution of the static robust model \hyperref[sro:2]{Static-Concurrent}, this implies the two models are equivalent.
\end{proof}

Based on Theorem \ref{thm:3}, the two-stage robust model \hyperref[2roc:1]{Two-stage-Concurrent} can be solved to optimality by solving its corresponding static robust model \hyperref[sro:2]{Static-Concurrent}. 

\subsection{The relationship between \hyperref[2ro:1]{Two-stage-Sequential} and \hyperref[2roc:1]{Two-stage-Concurrent}}

Until now, we have introduced two two-stage robust models \hyperref[2ro:1]{Two-stage-Sequential} and \hyperref[2roc:1]{Two-stage-Concurrent}, and also proved that these two models are equivalent to their corresponding static robust models respectively. In this subsection, we further investigate the relationships between \hyperref[2ro:1]{Two-stage-Sequential} and \hyperref[2roc:1]{Two-stage-Concurrent}.

First, we investigate the static models \hyperref[sro:2]{Static-Concurrent} and \hyperref[sro:1]{Static-Sequential}. Comparing static model \hyperref[sro:2]{Static-Concurrent} with static model \hyperref[sro:1]{Static-Sequential}, model \hyperref[sro:2]{Static-Concurrent} has less decision variables and less constraints and is computationally more attractive. The two static models are based on different recourse models. Next, we show that model \hyperref[sro:1]{Static-Sequential} and model \hyperref[sro:2]{Static-Concurrent} are equivalent with the support of the box uncertainty set.

\begin{thm}
\label{thm:4}
The static robust models \hyperref[sro:2]{Static-Concurrent} and \hyperref[sro:1]{Static-Sequential} are equivalent.
\end{thm}
\begin{proof}

First, suppose $(x_{ij}^*,x_{ijk}^*,y_i^*,z_k^*)$ is the optimal solution of the static robust model \hyperref[sro:1]{Static-Sequential}, and we introduce decision variable $y_{ij}$ which is described in \hyperref[eq:rc]{Recourse-Concurrent} model for model \hyperref[sro:1]{Static-Sequential}, then the optimal solution $(x_{ij}^*,x_{ijk}^*,y_i^*,z_k^*)$ can be mapped to an optimal solution $(x_{ij}^*,y_{ij}^*)$ of model \hyperref[sro:1]{Static-Sequential}, and it is clear that $(x_{ij}^*,y_{ij}^*)$ is feasible to the static robust model \hyperref[sro:2]{Static-Concurrent}.

Then, suppose $(x_{ij}^*,y_{ij}^*)$ is the optimal solution of the static robust model \hyperref[sro:2]{Static-Concurrent}, and we introduce variables $x_{ijk},y_i,z_k$ which is described in \hyperref[eq:rs]{Recourse-Sequential} model for model \hyperref[sro:2]{Static-Concurrent}, then the optimal solution $(x_{ij}^*,y_{ij}^*)$ can be mapped to an optimal solution $(x_{ij}^*,x_{ijk}^*,y_i^*,z_k^*)$ of model \hyperref[sro:2]{Static-Concurrent}, we now show that $(x_{ij}^*,x_{ijk}^*,y_i^*,z_k^*)$ is feasible to the static robust model  \hyperref[sro:1]{Static-Sequential}.

Suppose $z^*_{K'}=0$ and $z^*_{K'+1}=1$, then this means the nodes in the first-stage path become unreachable from the $(K'+1)$th node. Then, the length constraint \eqref{sroc:2} in the static robust model \hyperref[sro:2]{Static-Concurrent} is equivalent to
\begin{equation}
\sum_{k=1}^{K'}\sum_{(i,j)\in A}d_{ij}x^*_{ijk} + \sum_{(i,j)\in A}\bar{d}_{j0}x^*_{ijK'} \leq L,\forall \mathbf{d}\in\mathcal{U}\label{thm4:1}
\end{equation}

Based on the definition of the box uncertainty set $\mathcal{U}$, inequality \eqref{thm4:1} is equivalent to
\begin{equation}
\sum_{k=1}^{K'}\sum_{(i,j)\in A}(\bar{d}_{ij}+\Theta\hat{d}_{ij})x^*_{ijk} + \sum_{(i,j)\in A}\bar{d}_{j0}x^*_{ijK'} \leq L\label{thm4:2}
\end{equation}

We transform the left hand side of the above inequality as follows:
\begin{equation}
\begin{aligned}
&\sum_{k=1}^{K'}\sum_{(i,j)\in A}(\bar{d}_{ij}+\Theta\hat{d}_{ij})x^*_{ijk} + \sum_{(i,j)\in A}\bar{d}_{j0}x^*_{ijK'} \\
&=\sum_{k=1}^{K'-1}\sum_{(i,j)\in A}(\bar{d}_{ij}+\Theta\hat{d}_{ij})x^*_{ijk}+\sum_{(i,j)\in A}(\bar{d}_{ij}+\Theta\hat{d}_{ij})x^*_{ijK'} + \sum_{(i,j)\in A}\bar{d}_{j0}x^*_{ijK'}\\
&\geq\sum_{k=1}^{K'-1}\sum_{(i,j)\in A}(\bar{d}_{ij}+\Theta\hat{d}_{ij})x^*_{ijk}+\sum_{(i,j)\in A}(\bar{d}_{ij}+\bar{d}_{j0})x^*_{ijK'}\\
&\overset{1}{>} \sum_{k=1}^{K'-1}\sum_{(i,j)\in A}(\bar{d}_{ij}+\Theta\hat{d}_{ij})x^*_{ijk} + \sum_{(i,j)\in A}\bar{d}_{j0}x^*_{ij(K'-1)}\\
\end{aligned}
\label{thm4:3}
\end{equation}
where relation 1 is due to the triangle inequality.

Based on \eqref{thm4:2} and \eqref{thm4:3} we have 
\begin{equation}
\sum_{k=1}^{K'-1}\sum_{(i,j)\in A}(\bar{d}_{ij}+\Theta\hat{d}_{ij})x^*_{ijk} + \sum_{(i,j)\in A}\bar{d}_{j0}x^*_{ij(K'-1)} \leq L
\end{equation}
which implies 
\begin{equation}
\sum_{k=1}^{K'-1}\sum_{(i,j)\in A}d_{ij}x^*_{ijk} + \sum_{(i,j)\in A}\bar{d}_{j0}x^*_{ij(K'-1)} \leq L,\forall \mathbf{d}\in\mathcal{U}
\end{equation}

Following the above transformation recursively, finally we can get 
\begin{equation}
\sum_{k=1}^{K}\sum_{(i,j)\in A}d_{ij}x^*_{ijk} + \sum_{(i,j)\in A}\bar{d}_{j0}x^*_{ijK} \leq L,\forall K=1,...,K', \mathbf{d}\in\mathcal{U}
\end{equation}

So the optimal solution $(x_{ij}^*,x_{ijk}^*,y_i^*,z_k^*)$ is feasible to the static robust model  \hyperref[sro:1]{Static-Sequential}.

We conclude that the optimal solution of the static robust model \hyperref[sro:1]{Static-Sequential} is a feasible solution of the static robust model \hyperref[sro:2]{Static-Concurrent}, and the optimal solution of the static robust model \hyperref[sro:2]{Static-Concurrent} is a feasible solution of the static robust model \hyperref[sro:1]{Static-Sequential}. Then the two models have the same optimal solution, this implies the two models are equivalent.
\end{proof}

The following corollary shows the equivalence between two two-stage robust models.
\begin{cor}
\label{cor:1}
The two-stage robust models \hyperref[2ro:1]{Two-stage-Sequential} and \hyperref[2roc:1]{Two-stage-Concurrent} are equivalent.
\end{cor}
\begin{proof}
Based on Theorem \ref{thm:1}, Theorem \ref{thm:3} and Theorem \ref{thm:4}, we can draw this conclusion.
\end{proof}

Based on Theorems \ref{thm:1}-\ref{thm:4} and Corollary \ref{cor:1}, we know that the four models \hyperref[2ro:1]{Two-stage-Sequential}, \hyperref[sro:1]{Static-Sequential}, \hyperref[2roc:1]{Two-stage-Concurrent} and \hyperref[sro:2]{Static-Concurrent} are equivalent to each other. It is an interesting conclusion that \hyperref[2ro:1]{Two-stage-Sequential} and \hyperref[2roc:1]{Two-stage-Concurrent} are equivalent with the \emph{box uncertainty set} defined even though they are based on different recourse models. We can use the \hyperref[2roc:1]{Two-stage-Concurrent} model to deal with the two-stage OPSW with sequentially realized weights, and the \hyperref[2roc:1]{Two-stage-Concurrent} model is computationally more efficient than the \hyperref[2ro:1]{Two-stage-Sequential} model.

\section{Case study}
\label{section:cs}
In this section, a case study is presented to illustrate the effectiveness of the proposed two-stage robust models for OPSW. 

\subsection{Test instance}
The test instance used in our experiments is based on problem set 3 from \cite{tsiligirides1984heuristic} which was originally used for the deterministic OP. Problem set 3 contains 20 instances with the same 33 nodes and 20 varying length limits. We only consider 3 length limits: 80, 90 and 100. In the instance, the end point is ignored and the start point is kept as the depot location. The problem set can be found with URL: \url{http://www.mech.kuleuven.be/en/cib/op}.

To generate the uncertain instances for OPSW, we use the Euclidean distances between nodes as the expected weights $\bar{d}_{ij}$. Two kinds of uncertain instances are generated with the deviation values $\hat{d}_{ij}$ chosen as $0.2\bar{d}_{ij}$ and $0.5\bar{d}_{ij}$ respectively. Then, based on different length limits and different deviation values, we can get a total of 6 uncertain instances for OPSW. 

\subsection{Experiments}
In order to evaluate the effectiveness of the proposed two-stage robust models, we use one-stage robust model for OPSW as a comparison. In the one-stage robust model, all the decision variables are ``here and now'' and there are no recourse decision variables considered in the model. The one-stage robust model follows the traditional RO paradigm and is formulated as follows:

\begin{subequations}\label{eq:onestage}
(One-stage-RO):\begin{alignat}{2}
    \text{maximize}\qquad & \sum_{i\in N}s_i\sum_{j\in N^+\setminus\{i\}}x_{ij}\\
    \text{subject to }\qquad & \eqref{opc:2}-\eqref{opc:6}\\
    &\sum_{(i,j)\in A}d_{ij} x_{ij} \leq L,\forall \mathbf{d}\in\mathcal{U}
    \end{alignat}
  \end{subequations}

As proven in Section \ref{section:2ro}, the \hyperref[sro:2]{Static-Concurrent} model is equivalent to the two-stage robust models  \hyperref[2ro:1]{Two-stage-Sequential} and \hyperref[2roc:1]{Two-stage-Concurrent} and is computationally more efficient than the \hyperref[sro:1]{Static-Sequential} model. So we use \hyperref[sro:2]{Static-Concurrent} model to solve the two-stage OPSW with sequential or concurrent realized weights.

We generate 1000 scenarios for the $\hat{d}_{ij}=0.2\bar{d}_{ij}$ and $\hat{d}_{ij}=0.5\bar{d}_{ij}$ cases respectively for simulation purposes. Suppose the uncertain weights $\tilde{d}_{ij}$ are uniformly distributed on interval $[\bar{d}_{ij}-\hat{d}_{ij},\bar{d}_{ij}+\hat{d}_{ij}]$. Then the realizations $d_{ij}$ are sampled uniformly on the interval $[\bar{d}_{ij}-\hat{d}_{ij},\bar{d}_{ij}+\hat{d}_{ij}]$. 

For each uncertain instance, the \hyperref[sro:2]{Static-Concurrent} model and the \hyperref[eq:onestage]{One-stage-RO} model are solved by CPLEX 12.6 with $\Theta=0,0.1,...,1$ respectively. The robust solutions obtained by the \hyperref[sro:2]{Static-Concurrent} model and the \hyperref[eq:onestage]{One-stage-RO} model are then simulated with the 1000 scenarios for the two-stage OPSW with \hyperref[eq:rs]{Recourse-Sequential} and \hyperref[eq:rc]{Recourse-Concurrent} actions. The objective values of the second-stage recourse problems are calculated by Algorithm \ref{algorithm:1} and Algorithm \ref{algorithm:2}. The mean objective values and the standard deviations of the robust solutions are statistically summarized. 

\subsection{Numerical results}
Tables \ref{table:1}-\ref{table:6} show the numerical results of the 6 instances with different length limits and different deviation values. The Obj. in the tables represents the objective value obtained by the \hyperref[eq:onestage]{One-stage-RO} model or the \hyperref[sro:2]{Static-Concurrent} model. First we can observe that the objective values of the robust solutions are decreasing as parameter $\Theta$ increases for both the one-stage and two-stage RO models. As $\Theta$ increases, the size of the uncertainty set $\mathcal{U}$ is increasing which means the protection level is increasing, and the resulting robust solution is more conservative.

\begin{table*}[!ht]
      \small
      \caption{Numerical results of the instance with $L=80$ and $\hat{d}_{ij}= 0.2\bar{d}_{ij} $}
      \begin{center}
      \begin{tabular}{||c| c| c| c| c| c| c| c| c|c|c||}
      \hline
      \hline
      \multirow{3}{*}{$\Theta$} &
      \multicolumn{5}{|c|}{One-stage RO} &
      \multicolumn{5}{|c||}{Two-stage RO} \\
      \cline{2-6}
      \cline{7-11}
      &\multirow{2}{*}{Obj.} &
      \multicolumn{2}{|c|}{Sequential} &
      \multicolumn{2}{|c|}{Concurrent} &
      \multirow{2}{*}{Obj.} &
      \multicolumn{2}{|c|}{Sequential} &
      \multicolumn{2}{|c||}{Concurrent}\\
      \cline{3-6}
      \cline{8-11}
      &&Mean&Std.&Mean&Std.&&Mean&Std.&Mean&Std.\\
      \hline
  \hline
0.00 & 710.00 & 680.19 & 41.66 & 681.63 & 41.43 & 710.00 & \textbf{693.76} & \textbf{27.37} & \textbf{695.29} & \textbf{24.89} \\ 
   \hline
0.10 & 690.00 & 676.52 & 28.56 & 677.93 & 27.62 & 700.00 & \textbf{691.55} & \textbf{19.85} & \textbf{692.43} & \textbf{18.59} \\ 
   \hline
0.20 & 690.00 & \textbf{688.48} & \textbf{7.08} & \textbf{688.82} & \textbf{6.36} & 690.00 & 685.64 & 18.87 & 685.70 & 18.79 \\ 
   \hline
0.30 & 680.00 & 679.58 & 6.02 & 679.58 & 6.02 & 680.00 & \textbf{679.93} & \textbf{1.38} & \textbf{679.94} & \textbf{1.34} \\ 
   \hline
0.40 & 670.00 & 670.00 & \textbf{0.00} & 670.00 & \textbf{0.00} & 670.00 & \textbf{679.77} & 1.50 & \textbf{679.78} & 1.47 \\ 
   \hline
0.50 & 660.00 & \textbf{660.00} & \textbf{0.00} & \textbf{660.00} & \textbf{0.00} & 660.00 & \textbf{660.00} & \textbf{0.00} & \textbf{660.00} & \textbf{0.00} \\ 
   \hline
0.60 & 650.00 & \textbf{650.00} & \textbf{0.00} & \textbf{650.00} & \textbf{0.00} & 650.00 & \textbf{650.00} & \textbf{0.00} & \textbf{650.00} & \textbf{0.00} \\ 
   \hline
0.70 & 640.00 & 640.00 & \textbf{0.00} & 640.00 & \textbf{0.00} & 640.00 & \textbf{640.68} & 5.17 & \textbf{640.68} & 5.17 \\ 
   \hline
0.80 & 630.00 & 630.00 & 0.00 & 630.00 & 0.00 & 640.00 & \textbf{640.00} & \textbf{0.00} & \textbf{640.00} & \textbf{0.00} \\ 
   \hline
0.90 & 630.00 & \textbf{630.00} & \textbf{0.00} & \textbf{630.00} & \textbf{0.00} & 630.00 & \textbf{630.00} & \textbf{0.00} & \textbf{630.00} & \textbf{0.00} \\ 
   \hline
1.00 & 620.00 & 620.00 & 0.00 & 620.00 & 0.00 & 630.00 & \textbf{630.00} & \textbf{0.00} & \textbf{630.00} & \textbf{0.00} \\ 
  \hline
\hline
      \end{tabular}
      \end{center}
      \label{table:1}
      \end{table*}

\begin{table*}[!ht]
\small
      \caption{Numerical results of the instance with $L=80$ and $\hat{d}_{ij}= 0.5\bar{d}_{ij} $}
      \begin{center}
      \begin{tabular}{||c| c| c| c| c| c| c| c| c|c|c||}
      \hline
      \hline
      \multirow{3}{*}{$\Theta$} &
      \multicolumn{5}{|c|}{One-stage RO} &
      \multicolumn{5}{|c||}{Two-stage RO} \\
      \cline{2-6}
      \cline{7-11}
      &\multirow{2}{*}{Obj.} &
      \multicolumn{2}{|c|}{Sequential} &
      \multicolumn{2}{|c|}{Concurrent} &
      \multirow{2}{*}{Obj.} &
      \multicolumn{2}{|c|}{Sequential} &
      \multicolumn{2}{|c||}{Concurrent}\\
      \cline{3-6}
      \cline{8-11}
      &&Mean&Std.&Mean&Std.&&Mean&Std.&Mean&Std.\\
      \hline
  \hline
0.00 & 710.00 & \textbf{652.47} & \textbf{77.39} & \textbf{657.90} & \textbf{73.04} & 710.00 & \textbf{652.47} & \textbf{77.39} & \textbf{657.90} & \textbf{73.04} \\ 
   \hline
0.10 & 680.00 & 662.97 & 40.82 & 664.29 & 39.68 & 680.00 & \textbf{680.41} & \textbf{25.69} & \textbf{681.26} & \textbf{24.96} \\ 
   \hline
0.20 & 660.00 & \textbf{658.71} & \textbf{8.26} & \textbf{658.85} & \textbf{7.74} & 660.00 & 655.91 & 17.83 & 656.05 & 17.58 \\ 
   \hline
0.30 & 640.00 & \textbf{639.83} & \textbf{2.07} & \textbf{639.86} & \textbf{2.00} & 640.00 & 639.48 & 6.43 & 639.54 & 6.15 \\ 
   \hline
0.40 & 620.00 & 620.00 & \textbf{0.00} & 620.00 & \textbf{0.00} & 630.00 & \textbf{629.99} & 0.32 & \textbf{629.99} & 0.32 \\ 
   \hline
0.50 & 610.00 & 610.00 & \textbf{0.00} & 610.00 & \textbf{0.00} & 610.00 & \textbf{631.71} & 19.40 & \textbf{633.11} & 19.23 \\ 
   \hline
0.60 & 590.00 & 590.00 & 0.00 & 590.00 & 0.00 & 600.00 & \textbf{600.00} & \textbf{0.00} & \textbf{600.00} & \textbf{0.00} \\ 
   \hline
0.70 & 570.00 & 570.00 & 0.00 & 570.00 & 0.00 & 580.00 & \textbf{580.00} & \textbf{0.00} & \textbf{580.00} & \textbf{0.00} \\ 
   \hline
0.80 & 570.00 & 570.00 & 0.00 & 570.00 & 0.00 & 570.00 & \textbf{600.00} & \textbf{0.00} & \textbf{600.00} & \textbf{0.00} \\ 
   \hline
0.90 & 550.00 & 550.00 & \textbf{0.00} & 550.00 & \textbf{0.00} & 560.00 & \textbf{593.93} & 24.90 & \textbf{596.25} & 28.57 \\ 
   \hline
1.00 & 540.00 & 540.00 & 0.00 & 540.00 & 0.00 & 550.00 & \textbf{570.00} & \textbf{0.00} & \textbf{570.00} & \textbf{0.00} \\ 
  \hline
\hline
      \end{tabular}
      \end{center}
      \label{table:2}
      \end{table*}

\begin{table*}[!ht]
\small
      \caption{Numerical results of the instance with $L=90$ and $\hat{d}_{ij}= 0.2\bar{d}_{ij} $}
      \begin{center}
      \begin{tabular}{||c| c| c| c| c| c| c| c| c|c|c||}
      \hline
      \hline
      \multirow{3}{*}{$\Theta$} &
      \multicolumn{5}{|c|}{One-stage RO} &
      \multicolumn{5}{|c||}{Two-stage RO} \\
      \cline{2-6}
      \cline{7-11}
      &\multirow{2}{*}{Obj.} &
      \multicolumn{2}{|c|}{Sequential} &
      \multicolumn{2}{|c|}{Concurrent} &
      \multirow{2}{*}{Obj.} &
      \multicolumn{2}{|c|}{Sequential} &
      \multicolumn{2}{|c||}{Concurrent}\\
      \cline{3-6}
      \cline{8-11}
      &&Mean&Std.&Mean&Std.&&Mean&Std.&Mean&Std.\\
      \hline
  \hline
0.00 & 770.00 & \textbf{744.96} & \textbf{35.74} & \textbf{748.44} & \textbf{33.24} & 770 & 728.16 & 44.97 & 729.68 & 45.17 \\ 
   \hline
0.10 & 760.00 & 750.74 & 25.94 & 751.33 & 25.31 & 760 & \textbf{754.09} & \textbf{17.22} & \textbf{754.47} & \textbf{16.44} \\ 
   \hline
0.20 & 750.00 & 746.09 & 17.02 & 746.51 & 16.21 & 750 & \textbf{752.69} & \textbf{10.43} & \textbf{752.93} & \textbf{9.90} \\ 
   \hline
0.30 & 740.00 & 739.80 & \textbf{2.36} & 739.85 & \textbf{2.25} & 740 & \textbf{747.12} & 5.07 &\textbf{ 747.26} & 4.81 \\ 
   \hline
0.40 & 730.00 & \textbf{730.00} & \textbf{0.00} & \textbf{730.00} & \textbf{0.00} & 730 & \textbf{730.00} & \textbf{0.00} & \textbf{730.00} & \textbf{0.00} \\ 
   \hline
0.50 & 720.00 & \textbf{720.00} & \textbf{0.00} & \textbf{720.00} & \textbf{0.00} & 720 & \textbf{720.00} & \textbf{0.00} & \textbf{720.00} & \textbf{0.00} \\ 
   \hline
0.60 & 710.00 & \textbf{710.00} & \textbf{0.00 }& \textbf{710.00} & \textbf{0.00} & 710 & \textbf{710.00} & \textbf{0.00} & \textbf{710.00} & \textbf{0.00} \\ 
   \hline
0.70 & 700.00 & 700.00 & \textbf{0.00} & 700.00 & \textbf{0.00} & 710 & \textbf{736.43} & 4.83 & \textbf{736.71} & 4.74 \\ 
   \hline
0.80 & 690.00 &\textbf{ 690.00} & \textbf{0.00} & \textbf{690.00} & \textbf{0.00} & 690 & \textbf{690.00 }& \textbf{0.00} & \textbf{690.00} & \textbf{0.00} \\ 
   \hline
0.90 & 680.00 & 680.00 & 0.00 & 680.00 & 0.00 & 690 & \textbf{700.00} & \textbf{0.00} & \textbf{700.00} & \textbf{0.00} \\ 
   \hline
1.00 & 670.00 & 670.00 & 0.00 & 670.00 & 0.00 & 680 & \textbf{680.00} & \textbf{0.00} &\textbf{ 680.00} & \textbf{0.00} \\ 
  \hline
\hline
      \end{tabular}
      \end{center}
      \label{table:3}
      \end{table*}

\begin{table*}[!ht]
\small
      \caption{Numerical results of the instance with $L=90$ and $\hat{d}_{ij}= 0.5\bar{d}_{ij} $}
      \begin{center}
      \begin{tabular}{||c| c| c| c| c| c| c| c| c|c|c||}
      \hline
      \hline
      \multirow{3}{*}{$\Theta$} &
      \multicolumn{5}{|c|}{One-stage RO} &
      \multicolumn{5}{|c||}{Two-stage RO} \\
      \cline{2-6}
      \cline{7-11}
      &\multirow{2}{*}{Obj.} &
      \multicolumn{2}{|c|}{Sequential} &
      \multicolumn{2}{|c|}{Concurrent} &
      \multirow{2}{*}{Obj.} &
      \multicolumn{2}{|c|}{Sequential} &
      \multicolumn{2}{|c||}{Concurrent}\\
      \cline{3-6}
      \cline{8-11}
      &&Mean&Std.&Mean&Std.&&Mean&Std.&Mean&Std.\\
      \hline
  \hline
0.00 & 770.00 & \textbf{722.34} & \textbf{58.05 }& \textbf{727.46} & \textbf{57.50} & 770 & 719.97 & 59.25 & 725.56 & 58.42 \\ 
   \hline
0.10 & 740.00 & 718.99 & 48.56 & 721.56 & 43.38 & 740 & \textbf{730.86} & \textbf{38.69} & \textbf{733.30} & \textbf{36.30} \\ 
   \hline
0.20 & 720.00 & \textbf{718.37} & \textbf{10.40} & \textbf{718.43} & \textbf{10.37} & 720 & 716.06 & 16.54 & 716.74 & 15.18 \\ 
   \hline
0.30 & 690.00 & 689.98 & \textbf{0.45 }& 689.98 & \textbf{0.45} & 700 & \textbf{699.83} & 2.88 & \textbf{699.88} & 2.57 \\ 
   \hline
0.40 & 670.00 & 670.00 & \textbf{0.00} & 670.00 & \textbf{0.00} & 680 & \textbf{683.83} & 12.27 & \textbf{684.13} & 13.68 \\ 
   \hline
0.50 & 650.00 & 650.00 & \textbf{0.00} & 650.00 & \textbf{0.00} & 660 & \textbf{679.90} & 1.18 & \textbf{679.92} & 1.09 \\ 
   \hline
0.60 & 640.00 & 640.00 & \textbf{0.00} & 640.00 & \textbf{0.00} & 640 & \textbf{701.27} & 26.59 & \textbf{702.86} & 26.30 \\ 
   \hline
0.70 & 620.00 & 620.00 & 0.00 & 620.00 & 0.00 & 630 & \textbf{630.00} & \textbf{0.00} & \textbf{630.00} & \textbf{0.00} \\ 
   \hline
0.80 & 610.00 & 610.00 & \textbf{0.00} & 610.00 & \textbf{0.00} & 610 & \textbf{650.00} & 13.75 & \textbf{650.21} & 14.85 \\ 
   \hline
0.90 & 600.00 & 600.00 & \textbf{0.00 }& 600.00 & \textbf{0.00} & 600 & \textbf{649.83} & 26.63 & \textbf{652.74} & 31.74 \\ 
   \hline
1.00 & 580.00 & 580.00 & \textbf{0.00} & 580.00 & \textbf{0.00} & 590 & \textbf{604.88} & 13.09 & \textbf{608.52} & 16.38 \\ 
  \hline
\hline
      \end{tabular}
      \end{center}
      \label{table:4}
      \end{table*}

\begin{table*}[!ht]
\small
      \caption{Numerical results of the instance with $L=100$ and $\hat{d}_{ij}= 0.2\bar{d}_{ij} $}
      \begin{center}
      \begin{tabular}{||c| c| c| c| c| c| c| c| c|c|c||}
      \hline
      \hline
      \multirow{3}{*}{$\Theta$} &
      \multicolumn{5}{|c|}{One-stage RO} &
      \multicolumn{5}{|c||}{Two-stage RO} \\
      \cline{2-6}
      \cline{7-11}
      &\multirow{2}{*}{Obj.} &
      \multicolumn{2}{|c|}{Sequential} &
      \multicolumn{2}{|c|}{Concurrent} &
      \multirow{2}{*}{Obj.} &
      \multicolumn{2}{|c|}{Sequential} &
      \multicolumn{2}{|c||}{Concurrent}\\
      \cline{3-6}
      \cline{8-11}
      &&Mean&Std.&Mean&Std.&&Mean&Std.&Mean&Std.\\
      \hline
  \hline
0.00 & 800 & \textbf{795.39} & \textbf{7.86} & \textbf{795.68} & \textbf{7.84} & 800 & 776.88 & 45.66 & 778.99 & 41.13 \\ 
   \hline
0.10 & 800 & \textbf{798.06} & \textbf{5.57} & \textbf{798.14} & \textbf{5.53} & 800 & 797.02 & 7.11 & 797.20 & 7.05 \\ 
   \hline
0.20 & 790 & \textbf{789.83} & \textbf{1.86 }& \textbf{789.83} & 1.86 & 790 & 789.77 & \textbf{1.86} & 789.78 & \textbf{1.83} \\ 
   \hline
0.30 & 790 &\textbf{ 789.98} & \textbf{0.45 }& \textbf{789.99} & \textbf{0.32} & 790 & 789.89 & 1.70 & 789.89 & 1.70 \\ 
   \hline
0.40 & 780 & 780.00 & 0.00 & 780.00 & 0.00 & 790 & \textbf{790.00} & \textbf{0.00} & \textbf{790.00} & \textbf{0.00} \\ 
   \hline
0.50 & 780 & \textbf{780.00} & \textbf{0.00} & \textbf{780.00} & \textbf{0.00} & 780 & \textbf{780.00} &\textbf{ 0.00} & \textbf{780.00} & \textbf{0.00} \\ 
   \hline
0.60 & 760 & 760.00 & 0.00 & 760.00 & 0.00 & 770 & \textbf{770.00} &\textbf{ 0.00} & \textbf{770.00} & \textbf{0.00} \\ 
   \hline
0.70 & 760 & \textbf{760.00} & \textbf{0.00} & \textbf{760.00} & \textbf{0.00} & 760 & \textbf{760.00} & \textbf{0.00} & \textbf{760.00} & \textbf{0.00} \\ 
   \hline
0.80 & 740 & 740.00 & \textbf{0.00} & 740.00 & \textbf{0.00} & 750 & \textbf{769.39} & 2.39 & \textbf{769.48} & 2.22 \\ 
   \hline
0.90 & 740 & \textbf{740.00} & \textbf{0.00 }& \textbf{740.00} & \textbf{0.00} & 740 & \textbf{740.00} & \textbf{0.00} & \textbf{740.00} & \textbf{0.00} \\ 
   \hline
1.00 & 730 & \textbf{730.00} & \textbf{0.00} & \textbf{730.00} & \textbf{0.00} & 730 & \textbf{730.00} & \textbf{0.00} & \textbf{730.00} & \textbf{0.00} \\ 
  \hline
\hline
      \end{tabular}
      \end{center}
      \label{table:5}
      \end{table*}

\begin{table*}[!ht]
\small
      \caption{Numerical results of the instance with $L=100$ and $\hat{d}_{ij}= 0.5\bar{d}_{ij} $}
      \begin{center}
      \begin{tabular}{||c| c| c| c| c| c| c| c| c|c|c||}
      \hline
      \hline
      \multirow{3}{*}{$\Theta$} &
      \multicolumn{5}{|c|}{One-stage RO} &
      \multicolumn{5}{|c||}{Two-stage RO} \\
      \cline{2-6}
      \cline{7-11}
      &\multirow{2}{*}{Obj.} &
      \multicolumn{2}{|c|}{Sequential} &
      \multicolumn{2}{|c|}{Concurrent} &
      \multirow{2}{*}{Obj.} &
      \multicolumn{2}{|c|}{Sequential} &
      \multicolumn{2}{|c||}{Concurrent}\\
      \cline{3-6}
      \cline{8-11}
      &&Mean&Std.&Mean&Std.&&Mean&Std.&Mean&Std.\\
      \hline
  \hline
0.00 & 800 & 778.79 & 43.57 & 779.34 & 43.51 & 800 & \textbf{787.07} & \textbf{19.79} & \textbf{787.63} & \textbf{18.57} \\ 
   \hline
0.10 & 790 & 774.72 & 40.31 & 776.13 & 38.69 & 790 & \textbf{776.04} & \textbf{34.31} & \textbf{778.05} & \textbf{32.25} \\ 
   \hline
0.20 & 780 & 777.50 & 14.56 & 778.10 & 12.25 & 780 & \textbf{778.71} & \textbf{8.64} & \textbf{778.81} & \textbf{8.18} \\ 
   \hline
0.30 & 750 & 749.71 & 4.56 & 749.77 & 4.15 & 760 & \textbf{759.72} & \textbf{3.15 }& \textbf{759.79} & \textbf{2.84} \\ 
   \hline
0.40 & 730 & 730.00 & \textbf{0.00} & 730.00 & \textbf{0.00} & 730 & \textbf{749.27} & 3.57 & \textbf{749.33} & 3.44 \\ 
   \hline
0.50 & 710 &\textbf{ 710.00} & \textbf{0.00} & \textbf{710.00} & \textbf{0.00} & 710 & \textbf{710.00} & \textbf{0.00} & \textbf{710.00} & \textbf{0.00} \\ 
   \hline
0.60 & 690 & 690.00 & \textbf{0.00} & 690.00 & \textbf{0.00} & 690 & \textbf{719.75} & 2.29 & \textbf{719.81} & 2.02 \\ 
   \hline
0.70 & 670 & 670.00 & \textbf{0.00} & 670.00 & \textbf{0.00} & 670 & \textbf{718.05} & 38.05 & \textbf{720.08} & 38.58 \\ 
   \hline
0.80 & 650 & 650.00 & 0.00 & 650.00 & 0.00 & 660 & \textbf{720.00 }& \textbf{0.00} & \textbf{720.00} & \textbf{0.00} \\ 
   \hline
0.90 & 630 & 630.00 & 0.00 & 630.00 & 0.00 & 640 & \textbf{650.00} & \textbf{0.00} & \textbf{650.00} & \textbf{0.00} \\ 
   \hline
1.00 & 620 & 620.00 & 0.00 & 620.00 & 0.00 & 630 & \textbf{640.00} & \textbf{0.00} &\textbf{ 640.00} &\textbf{ 0.00} \\ 
  \hline
\hline
      \end{tabular}
      \end{center}
      \label{table:6}
      \end{table*}

For both the one-stage and two-stage robust models, the mean objective values with concurrent recourse are greater than or equal to the corresponding mean objective values with sequential recourse. The reason is that the concurrent recourse has more information on the uncertainty realizations than the sequential recourse, so the concurrent recourse can make a better recourse decision and achieve a lower loss of the collected score. However, the gaps between the mean objective values with sequential recourse and concurrent recourse are very small which means that the difference between the two recourse actions is small. 

Comparing the mean objective values of the one-stage RO and two-stage RO with the sequential recourse or concurrent recourse, the two-stage RO achieves better values than the one-stage RO in most cases. This is because that the two-stage RO considers the recourse decisions into the model but the one-stage RO only considers the first-stage decisions. We also notice that the one-stage RO can achieve better mean objective values than the two-stage RO in some cases. Figures \ref{fig:1}-\ref{fig:3} show visual comparisions between the one-stage RO and two-stage RO with sequential recourse. The comparisions between the one-stage RO and two-stage RO with concurrent recourse are visually similar. The figures clearly show that the two-stage RO dominates the one-stage RO in most cases, which show the effectiveness and superiority of the proposed two-stage robust models for dealing with the two-stage OPSW.

In Tables \ref{table:1}-\ref{table:6} we also report the standard deviations of the simulated robust solutions for both the one-stage RO and the two-stage RO. The standard deviations can reflect the stabilities of the obtained robust solutions. From the tables we can see that as parameter $\Theta$ increasing, the standard deviations tend to decrease, which means the robust solutions are more stable with a larger uncertainty set. We can also observe that the two-stage RO can mostly achieve better mean objective values and lower or small standard deviation values at the same time comparing with the one-stage RO. This further indicates that the proposed two-stage robust models can efficiently tackle the two-stage OPSW.

\begin{knitrout}
\definecolor{shadecolor}{rgb}{0.969, 0.969, 0.969}\color{fgcolor}\begin{figure}[!h]

{\centering \includegraphics[width=\maxwidth]{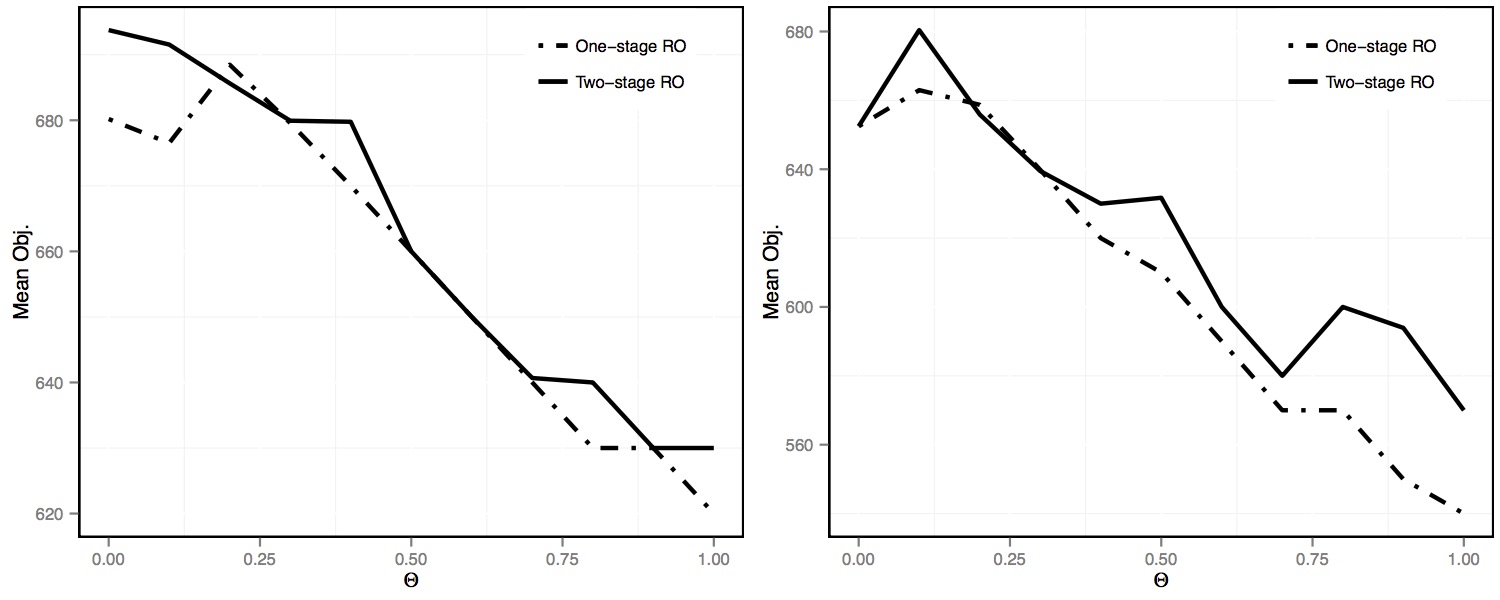} 

}

\caption{Comparision between the one-stage RO and two-stage RO with sequential recourse, $L=80$, $\hat{d}_{ij}= 0.2\bar{d}_{ij} $(left), $\hat{d}_{ij}= 0.5\bar{d}_{ij} $(right)}\label{fig:1}
\end{figure}

\end{knitrout}

\begin{knitrout}
\definecolor{shadecolor}{rgb}{0.969, 0.969, 0.969}\color{fgcolor}\begin{figure}[!h]

{\centering \includegraphics[width=\maxwidth]{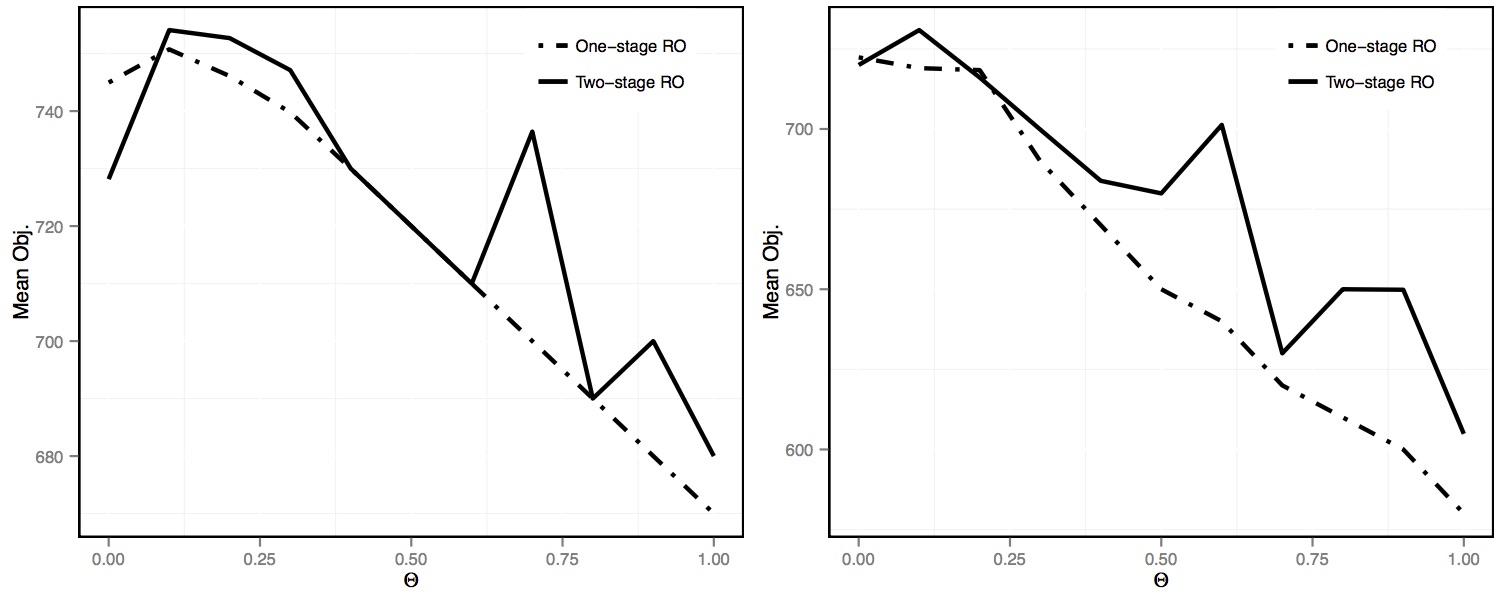} 

}

\caption{Comparision between the one-stage RO and two-stage RO with sequential recourse, $L=90$, $\hat{d}_{ij}= 0.2\bar{d}_{ij} $(left), $\hat{d}_{ij}= 0.5\bar{d}_{ij} $(right)}\label{fig:2}
\end{figure}

\end{knitrout}

\begin{knitrout}
\definecolor{shadecolor}{rgb}{0.969, 0.969, 0.969}\color{fgcolor}\begin{figure}[!h]

{\centering \includegraphics[width=\maxwidth]{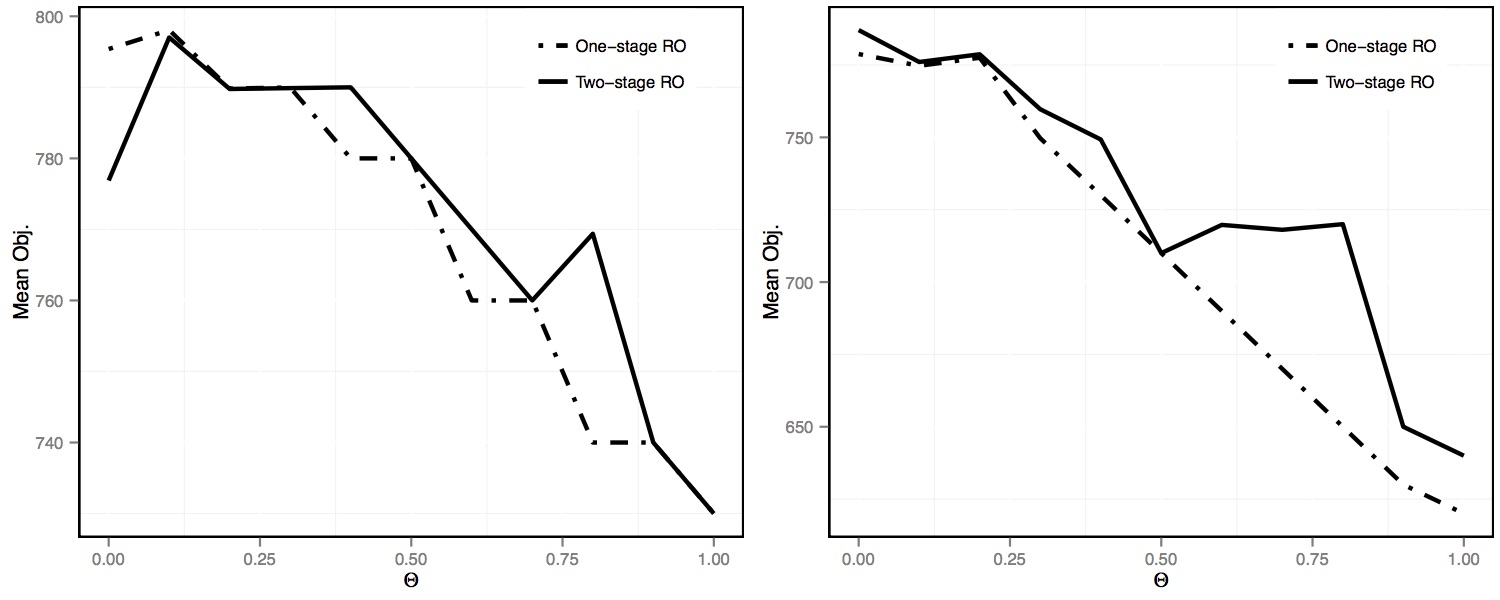} 

}

\caption{Comparision between the one-stage RO and two-stage RO with sequential recourse, $L=100$, $\hat{d}_{ij}= 0.2\bar{d}_{ij} $(left), $\hat{d}_{ij}= 0.5\bar{d}_{ij} $(right)}\label{fig:3}
\end{figure}

\end{knitrout}

\section{Conclusions}
\label{section:conclusion}
In this paper, we considered the OPSW with recourse actions. Based on different uncertainty realization ways, we presented two recourse models: one is the \hyperref[eq:rs]{Recourse-Sequential} model and the other is the \hyperref[eq:rc]{Recourse-Concurrent} model. The \hyperref[eq:rc]{Recourse-Concurrent} model has less decision variables and less constraints and is computationally more attractive. We applied the two-stage RO paradigm to the OPSW and introduced two two-stage RO models based on two recourse models. We theoretically proved that with the \emph{box uncertainty set} defined, the two-stage robust models are equivalent to their corresponding static robust models and the two two-stage robust models are also equivalent to each other. Subsequently, the two-stage robust models for OPSW can be determined to optimality by solving their corresponding static models. Comparative studies between the two-stage robust models and one-stage robust model for OPSW showed the effectiveness and superiority of the proposed two-stage robust models for tackling the two-stage OPSW.

We provide the following research directions as our future works:
\begin{enumerate}
\item The two-stage robust models for OPSW proposed in this paper are based on the \emph{box uncertainty set}, therefore we can draw theoretical conclusions on the equivalence between the two-stage robust models and their corresponding static robust models. Other uncertainty sets (e.g. the polyhedral uncertianty set) could be defined in the two-stage robust models and the performance of the corresponding static robust models can be studied.
\item The OPSW considered in this paper is with a two-stage setting where the decision variables are classified into two categories. As the planned path is executed dynamically and the nodes are visited sequentially, hence the OPSW can be viewed as a multi-stage decision making problem. So we can apply the multi-stage RO methodology and build a multi-stage robust model for the OPSW with a multi-stage setting.
\end{enumerate}

\section*{Conflict of Interests}
The authors declare that there is no conflict of interest regarding the publication of this manuscript.

\def\ack{\section*{Acknowledgments}%
  \addtocontents{toc}{\protect\vspace{6pt}}%
  \addcontentsline{toc}{section}{Acknowledgments}%
}
\ack{This work was supported by National Natural Science Foundation of China (No. 61573277, 71471158), the Research Grants Council of the Hong Kong Special Administrative Region, China (Project No. PolyU 15201414), the Fundamental Research Funds for the Central Universities, the Open Research Fund of the State Key Laboratory of Astronautic Dynamics under Grant 2015ADL-DW403, and the Scientific Research Foundation for the Returned Overseas Chinese Scholars, State Education Ministry, Natural Science Basic Research Plan in Shaanxi Province of China (No. 2015JM6316). The authors also would like to thank The Hong Kong Polytechnic University Research Committee for financial and technical support.}

\bibliographystyle{model1-num-names}
\bibliography{2ro.bib}

\end{document}